\documentclass[12pt,a4]{amsart}
\usepackage{amsmath,amsthm}
\usepackage[psamsfonts]{amssymb}
\usepackage{color}
\usepackage{epsfig}
\usepackage[colorlinks]{hyperref}
\usepackage{pgf,tikz}
\usepackage{graphicx}

\newtheorem{theorem}{Theorem}
\newtheorem{lemma}[theorem]{Lemma}

\theoremstyle{definition}

\newtheorem{example}[theorem]{Example}
\newtheorem{remark}[theorem]{Remark}
\theoremstyle{remark}

\newcommand{\R}{{\mathbb R}}

\newcommand{\dd}{\hskip0.2mm\mbox{\rm d}}

\numberwithin{equation}{section}
\begin{document}

\title[New results for the Liebau phenomenon]{New results for the Liebau phenomenon\\ via fixed point index}


\author{J. A. Cid}
\address{Jos{\'e} {\'A}ngel Cid, Departamento de Matem\'aticas, Universidade de Vigo,
32004, Pabell\'on 3, Campus de Ourense, Spain} \email{angelcid@uvigo.es}%

\author{G. Infante}
\address{Gennaro Infante, Dipartimento di Matematica e Informatica,
Universit\`{a} della Calabria, 87036 Arcavacata di Rende, Cosenza, Italy}
\email{gennaro.infante@unical.it}%

\author{M. Tvrd\'y}
\address{Milan Tvrd\'y, Mathematical Institute, Czech Academy of Sciences,
CZ~115~67 Praha~1, \v{Z}itn\'{a}~25, Czech Republic}
\email{tvrdy@math.cas.cz}%

\author{M. Zima}
\address{Miros\l awa Zima, Department of Differential Equations and Statistics,
Faculty of Mathematics and Natural Sciences, University of Rzesz\'ow,
Pigonia 1, 35-959 Rzesz{\'o}w, Poland.}
\email{mzima@ur.edu.pl}%

\begin{abstract}
We prove new results regarding the existence of positive solutions for a nonlinear
periodic boundary value problem related to the Liebau phenomenon. As a consequence we obtain
new sufficient conditions for the existence of a pump in a simple model. Our
methodology relies on the use of classical fixed point index. Some examples are
provided to illustrate our theory. We improve and complement previous results
in the literature.
\end{abstract}

\maketitle

\section{Introduction}

During the experiments developed in the 1950s, the German cardiologist Gerhart Liebau
observed (see \cite{liebau}) that a periodic compression could produce the circulation
of a fluid in a~mechanical system without valves to ensure the direction of the flow.
This valveless pumping effect is nowadays called the Liebau phenomenon. It was reported to occur
for instance in embryonic blood circulation, in applications of nanotechnology
and in oceanic currents, see e.g. \cite{and,borzipropst,bringley,mwy,propst}
or \cite[Chapter 8]{torres}. In particular, G.~Propst~\cite{propst} presented an explanation
of the pumping effect for flow configurations of several rigid tanks that are connected
by rigid pipes. He proved the existence of periodic solutions to the corresponding
differential equations for systems of 2 or 3 tanks. However, the apparently simplest
configuration consisting of 1 pipe and 1 tank turned out to be, from mathematical point
of view, the most interesting one, as it leads to the singular periodic problem
\begin{equation}\label{eqprobsing}
\left\{\begin{array}{l}
   u''(t)+a\,u'(t)=\displaystyle \frac 1{u(t)}\left(e(t)-b (u'(t))^2\right)-c,\quad t\in [0,T],
  \\[1mm]
   u(0)=u(T),\quad u'(0)=u'(T),
\end{array}\right.
\end{equation}
where $u'$ is the fluid velocity in the pipe (oriented in the direction from
the tank to the piston), $T>0,$
\[
  a=\frac{r_0}{\rho},\quad b=1\,{+}\,\frac{\zeta}{2},\quad
  c=\frac{g A_{\pi}}{A_{\tau}},\quad e(t)=\frac{g\,V_0}{A_\tau}{-}\frac{p(t)}{\rho}\,,
\]
$r_0$ is the friction coefficient, $\rho$ is the density of the fluid, $\zeta\ge 1$ is
the junction coefficient (depending on the particular geometry and smoothness of
the junction of the tank and the pipe), $g$ is the acceleration of gravity, $A_{\tau}$
is the cross section of the tank, $A_{\pi}$ is the cross section of the pipe (small
in comparison with $A_\tau$), $V_0$ is the total volume (assumed to be constant)
of the fluid in the system and $p$ is the $T$-periodic external force. As a result,
from the fluid mechanics point of view, the assumptions
\[
    a\ge 0,\quad b>1,\quad c>0, \quad e \mbox{\ continuous and\ } T-\mbox{periodic}
\]
are quite natural, from $\zeta\ge 1$ we would even have $b\ge 3/2$. Of course, we are interested in the search of positive solutions
of problem \eqref{eqprobsing}. A detailed justification of the model can be also found
e.g. in \cite[Chapter 8]{torres}.

One can observe that if a periodic external force $e$ produces a nonconstant periodic
response $u$ then the mean level of the fluid in the tank is higher than the level produced by
a~constant force with the same mean value. Moreover the increasing of the level is
proportional to $\|u'\|^2.$

The change of variables $u\,{=}\,x^{\,\mu},$ where $\mu\,{=}\,\frac 1{b\,{+}\,1},$ was used
in~\cite{cpt} in order to overcome the singularity, transforming and simplifying problem~\eqref{eqprobsing}
into the regular~BVP
\begin{equation}\label{eqprobreg}
\left\{\begin{array}{l}
  x''(t)+a\,x'(t)
   =\displaystyle\frac{e(t)}{\mu}\,x^{1\,{-}\,2 \mu}(t)-\frac{c}{\mu}\,x^{1\,{-}\,\mu}(t),
  \quad t\in [0,T],
 \\[2mm]
  x(0)=x(T),\quad x'(0)=x'(T),
\end{array}\right.
\end{equation}
where $0<\mu<\tfrac 12.$ By means of the lower and upper solution technique Cid and
co-authors~\cite{cpt} provided results on the existence and stability of a positive solution of \eqref{eqprobreg}.

In our recent paper~\cite{citz} we considered a generalization of problem~\eqref{eqprobreg},
namely
\begin{equation}\label{eqprob}
  \left\{\begin{array}{l}
     x''(t)+a x'(t)=r(t)\,x^{\alpha}(t)-s(t)\,x^{\beta}(t),\ t\in [0,T],
    \\[2mm]
     x(0)=x(T),\quad x'(0)=x'(T),
  \end{array}\right.
\end{equation}
under the assumption
\begin{equation}\tag{H0}
   a\ge 0, \quad  r,s:[0,T]\to\mathbb{R} \mbox{ \ are continuous and} \quad 0<\alpha<\beta<1.
\end{equation}

\noindent

Of course, to extend the obtained solution of the boundary value problem \eqref{eqprob} (on $[0,T]$) to a $T$-periodic solution of the corresponding differential equation, we would have to assume that $r$ and $s$ are also $T$-periodic. Making use of a shifting argument and Krasnosel'ski\u\i{}'s expansion/compression fixed point
theorem on cone, we succeeded in~\cite{citz} to improve the existence results from~\cite{cpt}.

Furthermore, Torres in~\cite[Chapter 8]{torres} obtained a priori bounds for the periodic
solutions of~\eqref{eqprobsing}, which together with the Brouwer degree theoretical arguments,
led to an alternative existence result.

We point out that the assumption
\[
    \min_{t\in [0,T]} e(t)>0
\]
is a common feature of the existence results in~\cite{torres,cpt,citz}. The goal of this paper
is twofold: first, to improve the main results from~\cite{citz} and, second, to obtain explicit
sufficient conditions for the existence of periodic solutions of problem~\eqref{eqprobreg} that allow, for
the first time, the function $e$ to be sign-changing. Similarly to the papers
\cite{gippmt,lan01,klamc04,li,jeff-2,jeff}, our main tool will be the classical fixed point
index on cones.

The paper is organized as follows: in Section 2 we present the shifting argument and recall some
known facts regarding the Green's function of the problem and some properties of the fixed point
index. In Section 3 we perform the fixed point index calculations that we use to prove our main
results. In Section 4 we present the main results, some consequences and illustrative examples.

\section{Preliminaries}
By means of a shifting argument (see~\cite{citz}) problem~\eqref{eqprob} may be rewritten in
the equivalent form
\begin{equation}\label{eqprob+m}
  \left\{\begin{array}{l}
     x''(t)+a\,x'(t)+m^2x(t)=r(t)\,x^{\alpha}(t)-s(t)\,x^{\beta}(t)+m^2x(t),\ t\in [0,T],
    \\[2mm]
     x(0)=x(T),\quad x'(0)=x'(T),
  \end{array}\right.
\end{equation}
with  $m\in\R.$ In the sequel, we denote the right-hand side of the differential equation
in \eqref{eqprob+m} by $f_m,$ i.e.
\begin{equation}\label{fm-def}
   f_m(t,x)=r(t)\,x^{\alpha}-s(t)\,x^{\beta}+m^2 x
   \quad\mbox{for \ } t\in[0,T] \mbox{ \ and\ } x\in [0,+\infty).
\end{equation}

We will assume that the linearization
\begin{equation}\label{eqlinnehom}
  \left\{\begin{array}{l}
     x''(t)+a\,x'(t)+m^2x(t)=h(t),\quad t\in [0,T],
    \\[2mm]
     x(0)=x(T),\quad x'(0)=x'(T),
  \end{array}\right.
\end{equation}
of \eqref{eqprob+m} possesses a positive Green's function. Its existence and its further
needed properties are given by the following lemma.

\begin{lemma}\label{green}
Assume that
\begin{equation}\label{antimax}
  a\ge 0 \mbox{ \ and \ } 0<m<\sqrt{\left(\frac{\pi}{T}\right)^2+\left(\frac{a}{2}\right)^2}\,.
\end{equation}
Then there exists a unique function $G_m(t,s)$ defined on $[0,T]\times[0,T]$ such that
\begin{align}\tag{G1}
   &\hskip-2mm\left\{\begin{array}{l}
      \mbox{For each\ } h \mbox{\ continuous on\ } [0,T], \mbox{\ the function\ }
     \\[2mm]
      \quad\displaystyle x(t)=\int_0^T G_m(t,s)\,h(s)\,\dd s\quad\mbox{for\ } t\in[0,T]
     \\[2mm]
      \mbox{is the unique solution of \eqref{eqlinnehom}}.
    \end{array}\right.
\\[2mm]\tag{G2}
   &G_m>0 \mbox{\ on\ } [0,T]\times[0,T].
\\[2mm]\tag{G3}
   &\int_0^T G_m(t,s)\,\dd s=\frac{1}{m^2} \mbox{ \ for all\ } t\in[0,T].
\\[2mm]\tag{G4}
   &\hskip-2mm\left\{\begin{array}{l}
      \mbox{There exists a constant\ } c_m\in (0,1) \mbox{\ such that\ }
     \\[1mm]
      G_m(t,s)\ge G_m(s,s)\ge c_m\,G_m(t,s)
      \mbox{ \ for all\ } (t,s)\in [0,T]\times[0,T]\,.
   \end{array}\right.
\end{align}
\end{lemma}
\proof
The existence of a function $G_m$ with the properties \thetag{G1} and \thetag{G2} is given
by \cite[Proposition 2.2]{omaritromb} and for the rest, see \cite[Appendix A]{citz}.
\endproof

Throughout, given $T$, $a$ and $m$ in \eqref{antimax}, we will write $G_m$ and $c_m$ for the corresponding function $G_m$ in (G1) and the corresponding constant $c_m$ in (G4).

Furthermore, let us recall that a \textit{cone} $P$ in a Banach space $X$ is a closed,
convex subset of $X$ such that $P\cap(-P)=\{0\}$ \ and \ $\lambda\,x\,{\in}\,P$ for $x\,{\in}\,P$
and $\lambda\,{\ge}\,0.$ Here we will work in the space $X\,{=}\,{\mathcal C}[0,T]$ of continuous
functions on $[0,T]$ endowed with the usual maximum norm $\|x\|=\max\{|x(t)|\,{:}\,t\in[0,T]\}.$
The properties~of Green's function $G_m$ of \eqref{eqlinnehom} stated in Lemma \ref{green}
enable us to use the cone
\begin{equation}\label{sposcone}
  P=\{x\in X\,{:}\,x(t)\ge c_m\,\|x\| \mbox{ \ for\ } t\in[0,T]\},
\end{equation}
a type of a cone first used by M.A.~Krasnosel'ski\u\i{} and D.~Guo, see for example
\cite{krzab, guolak}.

\vskip2mm

Whenever \eqref{antimax} is true, we define the operator $F\,{:}\,P\to X$ by
\begin{equation}\label{F-def}
  F\,x(t)=\int_0^T G_m(t,s)\,f_m(s,x(s))\,\dd s
  \quad\mbox{for \ } x\,{\in}\,P\mbox{ \ and \ } t\in[0,T]\,,
\end{equation}
where $G_m$ is the corresponding Green's function for \eqref{eqlinnehom}. Then, any fixed
point of $F$ in $P$ is a nonnegative solution of problem~\eqref{eqprob+m} and, simultaneously,
of problem \eqref{eqprob}. In order to obtain the existence of a fixed point of this kind
we make use of the classical fixed point index. For readers' convenience we formulate below
the principles that we will need later.

If $\Omega$ is an open bounded subset of $P$ (in the relative topology) we denote by
$\overline{\Omega}$ and $\partial\,\Omega$ the closure and the boundary relative to $P.$
When $\Omega$ is an open bounded subset of $X$ we write $\Omega_{P}=\Omega\,{\cap}\,P.$
Recall that for a given cone $P,$ an open set $\Omega$ and a compact operator
$F\,{:}\,\overline{\Omega_P}\to P,$ the symbol $i_{P}(F,\Omega_{P})$ stands for the fixed point
index of $F$ with respect to $P$ and $\Omega.$ For the definition and more details concerning
the properties of the fixed point index see e.g. \cite{guolak,amann, deimling, zeidler}. 
The following existence principle is well-known, but we include its short proof for completeness.
\begin{theorem}\label{ex-principle}
Let $P$ be a cone in a Banach space $X.$ Let $\Omega,\,\Omega'\subset X$ be open bounded
sets such that $0\in\Omega'_{P}$ and
\ $\overline{\Omega'_P}\subset\Omega_{P}.$ \ Assume that $F\,{:}\,\overline{\Omega_P}\to P$
is a compact map such that $F\,x\ne x$ for $x\in\partial\,\Omega_{P}\,{\cup}\,\partial\,\Omega'_P,$
\ $i_{P}(F,\Omega_P)=1$ and \ $i_{P}(F,\Omega'_P)=0.$
Then $F$ has a fixed point in $\Omega_P\setminus\overline{\Omega'_P}.$
\end{theorem}
\noindent{\it Proof} \ It follows from the additivity and solution properties of the fixed point index,
cf. e.g. \cite[Theorem 2.3.1 and Theorem 2.3.2]{guolak}. For an analogous argument, see also e.g.
the proof of Theorem 12.3 in \cite{amann}.
\hfill$\Box$

\vskip2mm

To verify the assumptions of Theorem~\ref{ex-principle}, the following assertion will
be helpful.

\begin{theorem}[\cite{guolak}, Lemma 2.3.1 and Corollary 2.3.1]\label{lemind}
Let $P$ be a cone in a Banach space $X,$ $\Omega$ be an open bounded set such that
$0\in\Omega_{P}$ and let $F\,{:}\,\overline{\Omega_P}\to P$
be compact. Then
\begin{itemize}
\item[(i)] \ If $F\,x\,{\ne}\,\lambda\,x$ for all $x\in\partial\,\Omega_{P}$ and all $\lambda\,{\ge}\,1,$
then $i_{P}(F,\Omega_{P})=1.$
\item[(ii)] \ If there exists $x_0\in P\,{\setminus}\,\{0\}$ such that \ $x\,{-}\,F\,x\,{\ne}\,\lambda\,x_0$
\ for all $x\in\partial\,\Omega_{P}$ and all $\lambda\,{\ge}\,0,$ then $i_{P}(F,\Omega_{P})=0.$
\end{itemize}
\end{theorem}

We will complete this section by introducing further notations needed later:

\vskip2mm

\noindent
For a given continuous function $h\,{:}\,[0,T]\to\R$ we denote
\begin{align*}
  &\overline{h}=\frac{1}{T}\int_0^Th(s)\,\dd s,\quad h_*=\min\{h(t)\,{:}\,t\in[0,T]\},\quad
  h^*=\max\{h(t)\,{:}\,t\in[0,T]\},
 \\[-3mm]\noalign{\noindent\mbox{and}}
  &h_+(t)=\max\{h(t),0\}\quad\mbox{for}\quad t\in[0,T].
\end{align*}

\section{Calculations of the fixed point index}

We will assume that \thetag{H0} holds and
\begin{equation}\tag{H1}
\left\{\begin{array}{l}
  \mbox{there are\ } m>0,\,R_2>0 \mbox{\ and\ } R_1\in(0,R_2) \mbox{\ such that\ }
\\[2mm]
  m^2<\left(\displaystyle\frac{\pi}{T}\right)^2+\left(\displaystyle\frac{a}{2}\right)^2
  \mbox{ \ and \ }
  f_m(t,x)\ge 0 \mbox{ \ for \ } t\in[0,T] \mbox{\ and\ } x\in[c_m\,R_1,R_2],
\end{array}\right.
\end{equation}
where $f_m$ is given by \eqref{fm-def}. Let $\widetilde{f}_m\,{:}\,[0,T]\times[0,R_2]\to\mathbb{R}$
be an arbitrary continuous and nonnegative function which coincides with $f_m$ on
$[0,T]\times[c_m\,R_1,R_2].$

For $\rho>0$ and $P$ as in \eqref{sposcone}, we define
\[
   B_\rho=\{x\,{\in}\,P\,{:}\, \|x\|\,{<}\,\rho\}
   \quad\mbox{and}\quad
   B'_\rho=\{x\,{\in}\,P\,{:}\, x_*\,{<}\,c_m\,\rho\}.
\]
Let us note that the sets of the form $B'_\rho$ were introduced and utilized already by Lan
in \cite{lan01}. Note that we have $B'_\rho\subset B_\rho.$

We will consider the operator
\begin{equation}\label{wtF}
  \widetilde{F}x(t)=\int_0^T G_m(t,s)\,\widetilde{f}_m(s,x(s))\,\dd s
  \quad\mbox{for \ } x\in\overline{B_{R_2}}\,.
\end{equation}
where $G_m$ is Green's function of \eqref{eqlinnehom} whose existence and properties are
given by Lemma~\ref{green}.

The main object of this Section is to yield sufficient conditions for the fixed point
index of $\widetilde{F}$ to be 1 or 0. In order to do this, we utilize, in the spirit
of the papers by Lan~\cite{klamc04} and Webb ~\cite{jeff}, the explicit dependence
of the nonlinearity $f_m$ on $t.$

\begin{lemma}\label{Fpositive}
Assume \thetag{H0} and \thetag{H1}. Then the operator $\widetilde{F}$ maps $\overline{B_{R_2}}$
into $P$ and is compact.
\end{lemma}
\begin{proof}
For $x\in\overline{B_{R_2}}$ and $t\in[0,T],$ we have by Lemma \ref{green}
and construction of $\widetilde{f}_m$
\[
  \widetilde{F}x(t)=\int_0^TG_m(t,s)\,\widetilde{f}_m(s,x(s))\,\dd s
  \ge\int_0^TG_m(s,s)\,\widetilde{f}_m(s,x(s))\,\dd s
\]
and
\[
  \int_0^TG_m(s,s)\,\widetilde{f}_m(s,x(s))\,\dd s
  \ge c_m\int_0^TG_m(t,s)\,\widetilde{f}_m(s,x(s))\,\dd s.
\]
Therefore $\widetilde{F}x(t)\ge c_m\|\widetilde{F}x\|$ on $[0,T].$ Hence
$\widetilde{F}(\overline{B_{R_2}})\,{\subset}\,P.$
The compactness of $\widetilde{F}$ follows in a~standard way from the Arzel\`{a}-Ascoli theorem.
\end{proof}

\begin{lemma}\label{le-ind-1c}
Assume that \thetag{H0} and \thetag{H1} hold, $\widetilde{F}\,x\,{\ne}\,x$ for
$x\,{\in}\,\partial B_{R_2}$ and there exists a~continuous function $g_1$
such that
\begin{align}\tag{H2}
    &f_m(t,x)\le g_1(t) \quad\mbox{for\ } t\in[0,T] \mbox{\ and\ } x\in[c_m\,R_2,R_2]
  \\\noalign{\noindent\mbox{and moreover either}}\tag{H3}
   &\delta_*\le c_m\,R_2,
\\[-2mm]\noalign{\noindent\mbox{or}}\tag{H4}
   &\delta^*\le  R_2,
\\[-2mm]\noalign{\noindent\mbox{where}}\nonumber
   &\delta(t)=\int_0^T G_m(t,s)\,g_1(s)\,\dd s \quad\mbox{for \ } t\in [0,T].
\end{align}
Then $i_P(\widetilde{F},B_{R_2})=1.$
\end{lemma}
\begin{proof} It follows from Lemma \ref{Fpositive} that $\widetilde{F}$ maps
$\overline{B_{R_2}}$ into $P$ and is compact. We show that $\widetilde{F}x\ne\lambda\,x$
for $x\in\partial B_{R_2}$ and $\lambda\,{\ge}\,1,$ which by Theorem \ref{lemind}~(i)
(where we put $\Omega_P\,{=}\,B_{R_2}$) yields $i_P(\widetilde{F},B_{R_2})\,{=}\,1.$

If not, there exist $\lambda\,{\ge}\,1$ and $x\,{\in}\,P$ with $\|x\|=R_2$ such that
$\widetilde{F}\,x=\lambda\,x.$ Since we suppose that $\widetilde{F}\,x\ne x$ for
$x\,{\in}\,\partial B_{R_2},$ we may assume that $\lambda\,{>}\,1.$ Consequently,
$c_m\,R_2\,{\le}\,x(t)\,{\le}\,R_2$ for $t\,{\in}\,[0,T],$ and, since $G_m(t,s)\,{>}\,0,$
hypothesis (H2) implies
\begin{equation*}
  \lambda\,x(t)=\int_0^T G_m(t,s)\,\widetilde{f}_m(s,x(s))\,\dd s
  \le\int_0^T G_m(t,s)\,g_1(s)\,\dd s=\delta(t).
\end{equation*}
If \thetag{H3} holds then we obtain
\[
   \lambda\,c_m\,R_2\le\lambda\,x_*
  =\min_{t{\in}[0,T]}\int_0^T G_m(t,s)\,\widetilde{f}_m(s,x(s))\,\dd s
  \le\delta_*\le c_m\,R_2\,,
\]
a contradiction since $\lambda>1.$ If \thetag{H4} holds, then taking the maximum on $[0,T]$
we obtain
\ $\lambda\,R_2=\lambda\,x^*\le\delta^*\le R_2,$ \ a contradiction.
\end{proof}

\begin{lemma}\label{le-ind-0} Assume that \thetag{H0} and \thetag{H1} hold, $\widetilde{F}\,x\,{\ne}\,x$
for $x\in\partial B'_{R_1}$ and there exists a~continuous function $g_0$ such that
\begin{align}\tag{H5}
    &f_m(t,x)\ge g_0(t)\ge 0 \mbox{ \ for\ } t\in [0,T] \mbox{ \ and\ } x\in [c_m\,R_1,R_1],
  \\\noalign{\noindent\mbox{and moreover either}}\tag{H6}
    &\gamma_*\ge c_m\,R_1,
  \\\noalign{\noindent\mbox{or}}\tag{H7}
    &\gamma^*\ge  R_1,
  \\\noalign{\noindent\mbox{where}}\nonumber
    &\gamma(t)=\int_0^T G_m(t,s)\,g_0(s)\,\dd s \quad\mbox{for \ } t\in [0,T].
\end{align}
Then $i_{P}(\widetilde{F},B'_{R_1})=0.$
\end{lemma}
\begin{proof} By Lemma~\ref{Fpositive}, the operator
\ $\widetilde{F}\,{:}\,\overline{B'_{R_1}}\,{\subset}\,\overline{B_{R_1}}\to P$ is compact.
Put $x_0(t)\,{=}\,1$ for $t\in [0,T].$ Then $x_0\,{\in}\,P$ and we shall prove that
$x\,{-}\,\widetilde{F}\,x\,{\ne}\,\lambda\,x_0$ for all $x\in\partial B'_{R_1}$ and
$\lambda\,{\ge}\,0,$ which by Theorem~\ref{lemind}~(ii) (where we put $\Omega'_P\,{=}\,B'_{R_1}$)
implies that $i_{P}(\widetilde{F},B'_{R_1})=0.$

If not, there exist $x\in\partial B'_{R_1}$ and $\lambda\ge 0$ such that
$x=\widetilde{F}\,x+\lambda\,x_0.$ Due to our assumption  $\widetilde{F}\,x\,{\ne}\,x$
for $x\in\partial B'_{R_1},$ it is enough to consider $\lambda>0.$ Furthermore, note that
$x\in\partial B'_{R_1}$ means that $x\,{\in}\,P$ and $x_*\,{=}\,c_m\,R_1.$

In particular, $x\in\partial B'_{R_1}$ implies that $c_m\,R_1\le x(t)\le R_1$ for $t\in [0,T].$
Then due to \thetag{H5} and the positivity of Green's function $G_m$ we have for $t\in [0,T]$
that
\begin{align*}
  x(t)&=\widetilde{F}\,x(t)+\lambda x_0(t)
  =\int_0^T G_m(t,s)\widetilde{f}_m(s,x(s))\,\dd s+\lambda
 \\
  &\ge\int_0^T G_m(t,s)\,g_0(s)\,\dd s+\lambda=\gamma(t)+\lambda.
\end{align*}
Now, if \thetag{H6} holds then, using Lemma~\ref{green} we get the contradiction
\[
   c_m\,R_1=x_*\ge\gamma_*+\lambda>\gamma_*\ge c_m\,R_1.
\]
On the other hand, if \thetag{H7} holds, using again Lemma~\ref{green} we arrive
at the contradiction
\[
  R_1\ge x^*\ge\gamma^*+\lambda>\gamma^*\ge R_1.
\]
\vskip-8mm\end{proof}

\begin{remark}\label{Rem7}
Changing slightly the proof of Lemma \ref{le-ind-0} we could prove that
\ $i_{P}(\widetilde{F},B_{R_1})=0$ holds also under (H0),(H1),(H5) and (H7).
Furthermore, it is worth mentioning that (H7) is equivalent to
\begin{equation}\tag{H7'}
   \mbox{there exists \ } t_0\in[0,T] \mbox{ \ such that\ } \gamma(t_0)\ge R_1,
\end{equation}
a condition that is, in general, more readily verifiable.
\end{remark}

\section{Main Results}

Now we are ready to formulate and prove our main results.

\begin{theorem}\label{main}
If the hypotheses \thetag{H0}-\thetag{H3}, \thetag{H5} and \thetag{H6} hold, then
problem~\eqref{eqprob} has at least one solution $x$ such that $c_m\,R_1\le x(t)\le R_2$ on $[0,T].$
The same conclusion remains true with \thetag{H3} replaced by \thetag{H4} and/or \thetag{H6}
replaced by \thetag{H7}.
\end{theorem}
\begin{proof}
Let the operators $F$ and $\widetilde{F}$ be respectively given by \eqref{F-def}
and \eqref{wtF}.  Recall that $F$ and $\widetilde{F}$ coincide on
$\overline{B_{R_2}}\setminus B'_{R_1}$ and thus each fixed point of $\widetilde{F}$ in $\overline{B_{R_2}}\setminus B'_{R_1}$ is a solution of \eqref{eqprob}.

Now, if $\widetilde{F}\,x=x$ for some $x\in\partial\,B_{R_2}\cup\partial\,B'_{R_1},$ then
this $x$ is a solution to our problem. On the other hand, if $\widetilde{F}\,x\ne x$ for all
$x\in\partial\,B_{R_2}\cup\partial\,B'_{R_1},$ then, $i_P(\widetilde{F},B_{R_2})\,{=}\,1$
by Lemma~\ref{le-ind-1c} and $i_P(\widetilde{F},B'_{R_1})\,{=}\,0$ by Lemma \ref{le-ind-0}.
Therefore, by Theorem \ref{ex-principle} (where we put $\Omega_P\,{=}\,B_{R_2}$ and
$\Omega'_P\,{=}\,B'_{R_1}$), the operator $\widetilde{F}$ has a fixed point in
$B_{R_2}\,{\setminus}\,\overline{B'_{R_1}}.$ To summarize, problem \eqref{eqprob} has
a solution $x\in\overline{B_{R_2}}\,{\setminus}\,B'_{R_1}.$ Finally, notice that
$x\,{\in}\overline{B_{R_2}}\,{\setminus}\,B'_{R_1}$ if and only if $x\,{\in}\,P$
and $c_m\,R_1\le x(t)\le R_2$ on $[0,T].$ This completes the proof of the theorem.
\end{proof}

\begin{remark}\label{rem-main}
We would like to emphasize that the assumptions of Theorem \ref{main} are weaker than those
of Theorem 3.2 of \cite{citz}. In particular, if $f_m(t,x)\le m^2R_2$ for $t\in[0,T]$ and
$x\in [c_m\,R_2,R_2]$ then (H2) and (H4) of Theorem \ref{main} are satisfied with $g_1(t)=m^2R_2.$
Clearly, hypotheses (H2)--(H4) also hold if there exists a continuous function $\hat{g},$
such that $f_m(t,x)\le\hat{g}(t)R_2$ for $t\in[0,T]$ and $x\in[c_m\,R_2,R_2]$ with
\begin{align*}
   &\min_{t\in [0,T]}\int_0^T G_m(t,s)\,\hat{g}(s)\,\dd s\le c_m
  \\\noalign{\noindent\mbox{or}}
   &\max_{t\in [0,T]}\int_0^T G_m(t,s)\,\hat{g}(s)\,\dd s\le 1.
\end{align*}
Moreover, if $f_m(t,x)\ge m^2R_1$ for $t\in[0,T]$ and $x\in[c_m\,R_1,R_1]$ then
\thetag{H5}-\thetag{H7} of Theorem \ref{main} are fulfilled with $g_0(t)=m^2\,R_1.$
For similar type of comparisons see~\cite{klamc04, jeff}.
\end{remark}

\begin{remark} Notice that the nature of our approach does not allow us to derive direct conclusions regarding the uniqueness of the solution in Theorem \ref{main}. On the other hand, Theorem \ref{main} provides a localization $x\in\overline{B_{R_2}}\,{\setminus}\,B'_{R_1}$
of the corresponding solution $x$ to problem~\eqref{eqprob}, which means that
\begin{equation}\label{local}
   x\,{\in}\,P \quad\mbox{and}\quad c_m\,R_1\le x(t)\le R_2 \mbox{ \ for\ } t\in[0,T].
\end{equation}
Having in mind Remark~\ref{Rem7}, it is not difficult to modify the proof of Theorem \ref{main}
so that, under the assumptions \thetag{H0}, \thetag{H1}, \thetag{H2}, \thetag{H3}
(or \thetag{H4}), \thetag{H5} and \thetag{H7}, we obtain the existence of a solution
to problem \eqref{eqprob} in the set $\overline{B_{R_2}}\setminus B_{R_1},$ where
in comparison with the localization from Theorem~\ref{main}, the set $B'_{R_1}$
is replaced by $B_{R_1}.$ Observe that this means that the corresponding solution $x$
of the given problem will satisfy in addition to \eqref{local} also the condition $x^*\ge R_1.$
\end{remark}

Next, we will apply Theorem \ref{main} to provide  sufficient conditions for the existence
of positive solutions of problem~\eqref{eqprobreg}, which is a special case of \eqref{eqprob} with
\[
   r(t)=\frac{e(t)}{\mu},\quad
   s(t)=\frac{c}{\mu},\quad
   \alpha=1-2\mu \mbox{ \ and \ } \beta=1-\mu\,.
\]

Recall that it was proved in~\cite{cpt} that the necessary condition for the existence of a positive
solution of~\eqref{eqprobreg} is $\overline{e}>0,$ and all existence theorems known up to now
(cf. \cite{torres,cpt,citz}) required $e$ to be strictly positive on $[0,T].$ However, it is easy
to verify that for
\begin{equation}\label{example}
    e(t)=0.1 V_0 +1.8+(2.1-V_0) \cos\,t-3 \cos^2t,
    T=2\,\pi,  a=0, b=2  \mbox{\ and\ } c=0.1,
\end{equation}
with $V_0>3,$ we have $\overline{e}>0$ and $e_*<0,$  while
\[
 u(t)=(V_0-2)+\cos\,t,
\]
is a solution to the corresponding problem \eqref{eqprobsing} (Example \eqref{example} is in fact,
a slight modification of the example by G. Propst from \cite[\thetag{19}]{propst}).This indicates
that the positivity of $e$ can not be a necessary condition for the existence of positive
solutions to problem \eqref{eqprobsing} and it should be weakened. To our knowledge, next existence
principle allows to deal for the first time also with the case $e_*\le 0.$

\begin{theorem}\label{Liebau}
Suppose that
\begin{equation}\tag{C0}
   a\ge0, \,\, 0<\mu<\tfrac 12, \,\,  c>0, \,\, T>0, \mbox{ \ and\ }
   e\,{:}\,[0,T]\to\mathbb R \mbox{\ is continuous}.
\end{equation}
Moreover, assume that there are $m\,{>}\,0,\, \kappa\,{>}\,0,$ and $R_1,R_2\,{>}\,0$ with
$R_1\,{<}\,R_2$ such that \eqref{antimax} holds and
\begin{align}\tag{C1}
    &e_*\le 0 \mbox{ \ and \ }
    (c_m\,R_1)^{\,\mu}\ge\displaystyle\frac{c+\sqrt{c^2-4\,\mu\,m^2 e_*}}{2\mu m^2}\,,
 \\\tag{C2}
   &(c_m R_1)^{1-2\mu}\ge\kappa\,\mu\,,
 \\\tag{C3}
   &\int_0^T G_m(s,s)\,e_+(s)\,\dd s\ge\frac{c_m\,R_1}{\kappa}\,,
 \\\tag{C4}
   &e^*\le c\,R_2^{\,\mu}\,.
\end{align}

Then there exists a positive solution of problem~\eqref{eqprobreg}.
\end{theorem}
\begin{proof}
We have
$f_m(t,x)=\displaystyle\frac{e(t)}{\mu}\,x^{1\,{-}\,2\,\mu}
          -\displaystyle\frac{c}{\mu}\,x^{1\,{-}\,\mu}+m^2 x.$
 We will check the hypotheses (H0)--(H2), (H4)--(H6) of Theorem \ref{main}.
Put $g_1(t)=m^2 R_2$ and $g_0(t)=\kappa\,e_+(t)$ for $t\in[0,T].$

Clearly, \thetag{H0} follows from \thetag{C0}. We will show that \thetag{H1} holds, as well.
Indeed, (C1) gives $\mu\,m^2 x^{2\mu}-c\,x^{\,\mu}+e_*\ge 0$ on $[c_m\,R_1,R_2].$ Hence
$x^{1-2\mu}\,(\mu\,m^2 x^{2\mu}\,{-}\,c\,x^{\,\mu}\,{+}\,e_*)\ge 0$ on $[c_m\,R_1,R_2].$
Consequently, we have
\[
  f_m(t,x)\ge\frac{e_*}{\mu}\,x^{1-2\,\mu}-\frac{c}{\mu}\,x^{1-\mu}+m^2 x\ge 0
  \quad\mbox{for \ } t\in[0,T] \mbox{ \ and \ } x\in[c_m R_1,R_2].
\]
This, together with \eqref{antimax} means that (H1) holds.

Now, we will show that \thetag{H2} follows from $\mu\in (0,\tfrac 12),$ \thetag{C1} and \thetag{C4}.
We need to prove that
\[
  e^*\le\mu\,m^2 R_2\,x^{2\mu{-}1}\,{+}\,c\,x^{\,\mu}\,{-}\,\mu\,m^2 x^{2\mu}
\]
for $x\in [c_m R_2,R_2].$ The function $\mu\,m^2 R_2\,x^{2\mu-1}+c  x^{\,\mu}-\mu\,m^2 x^{2\mu}$
is decreasing on $[c_m R_2,R_2],$ since by $\mu\in (0,\tfrac 12)$ and (C1) we get
\[
  (\mu\,m^2 R_2\,x^{2\mu-1}+c\,x^{\,\mu}-\mu\,m^2 x^{2\mu})'
  =\mu\,x^{\mu-1}\,((2\mu-1)\,m^2R_2\,x^{\mu-1}+c-2\mu\,m^2 x^{\,\mu})<0\,.
\]
From \thetag{C4} we obtain
\[
   e^*\le c\,R_2^{\,\mu}\le\mu\,m^2R_2\,x^{2\mu-1}+c\,x^{\,\mu}-\mu\,m^2 x^{2\mu},
\]
which proves \thetag{H2}. Due to Lemma~\ref{green}, $\delta^*=R_2$ in this case. Hence,
\thetag{H4} is fulfilled, as well.

Let $t\in[0,T]$ be such that $e(t)<0.$ Then, since $e_+(t)=0$ in this case, the inequality in \thetag{H5} obviously holds. Next, let $t\in[0,T]$
be such that $e(t)\ge 0.$ Then we need to show that
\begin{equation}\label{eq-e}
   e(t)\,(x^{1-2\mu}-\kappa\,\mu)\ge x^{1-\mu}\,(c-\mu\,m^2 x^{\,\mu})
\end{equation}
for $x\in [c_m\,R_1,R_1].$ Since
\[
  \frac{c+\sqrt{c^2-4\mu\,m^2 e_*}}{2\mu\,m^2}\ge\frac{c}{\mu\,m^2}\,,
\]
from \thetag{C1} and \thetag{C2} we get $c-\mu\,m^2 x^{\,\mu}\le 0$ and $x^{1-2\mu}-\kappa\,\mu\,\ge 0$
for $x\in [c_m\,R_1,R_1].$ This gives \eqref{eq-e}, and therefore \thetag{H5} holds.
\thetag{H6} follows immediately from \thetag{C3}.
\end{proof}

\begin{remark} Notice that in Theorem \ref{Liebau} the smaller are the difference $e^*-e_*$ and the period $T,$
the greater is the chance to find proper  $m,\kappa,R_1$ and $R_2.$

Moreover, one can verify that example \eqref{example} mentioned above does not satisfy the assumptions
of Theorem \ref{Liebau}. Possible extensions of our existence principle so that they would cover also
such examples remain an open problem.
\end{remark}

Next example is an illustrative application of Theorem \ref{Liebau}.

\begin{example}\label{ex-neg} Let us consider problem \eqref{eqprobreg} with the parameter
values $a=1.6,$ $\mu=0.01,$ $c=0.005,$ $T=1$ and being $e$ the continuous periodic function
defined on $[0,1]$ as follows
\[
  e(t)=
  \begin{cases}
    e_*+\dfrac{e^*-e_*}{t_1}\,t \quad \hbox{for}\quad t\in[0,t_1),\\
    e^*\quad \hbox{for}\quad t\in[t_1,t_2),\\
    e^*-\dfrac{e^*-e_*}{1-t_2}(t-t_2) \quad \hbox{for}\quad t\in[t_2,1],
  \end{cases}
\]
where $t_1=0.0005,$ $t_2=0.9995,$ $e^*=0.00548239$ and $e_*=-0.00005.$

Now, defining $m=0.7,$ $R_1=25,$ $R_2=10000$ and $\kappa=2200,$ direct computations show that all
the assumptions of Theorem \ref{Liebau} are satisfied and therefore there exists a positive solution
of problem~\eqref{eqprobreg} (note that, using explicit formulas from Appendix A.1 of \cite{citz} and the Mathematica software, we can
for $a=1.6$ and $m=0.7$ approximate the value of $c_m$ by 0.94144. Similarly we get $G_m(s,s)\approx 1.96026$).
\end{example}

In \cite[Example 3.7]{citz} we showed that our main result of \cite{citz}, namely Theorem 3.2,
is more general than  \cite[Theorem 1.8]{cpt} (see also \cite[Corollary 3.6]{citz}) when applied
to problem~\eqref{eqprobreg}. However, since the function $e(t)\equiv 1.54$ used
in Example 3.7 is constant, problem~\eqref{eqprobreg} has a constant solution, namely
\[
   x(t)=\left(\frac{e(t)}{c}\right)^\frac{1}{\mu}\equiv \left(\frac{1.54}{1.49}\right)^{100}=27.1297.
\]
Although \cite[Example 3.7]{citz} is correct, the case of $e$ constant is not meaningful for the pumping effect in the studied configuration
(see \cite[Definition 1.2]{cpt}). Now we present an application of Theorem \ref{main}
to problem~\eqref{eqprobreg} that allows a nonconstant positive function $e$ which, as it will
be shown in Example \ref{ex-mod}, does not satisfy the assumptions of \cite[Theorem 1.8]{cpt}.

\begin{theorem}\label{th}
Suppose that \thetag{C0} of Theorem $\ref{Liebau}$ holds, and moreover:
\begin{align}\tag{C5}
     &e_*>0,
   \\\tag{C6}
     &\mbox{there is \ } m>0 \mbox{\ such that \ }
      \frac{c^2}{\mu\,e_*^2}\left(\frac{1}{c_m^\mu}\,e^*-e_*\right)
      \le m^2<\left(\frac{\pi}{T}\right)^2+\left(\frac{a}{2}\right)^2.
\end{align}

Then problem~\eqref{eqprobreg} has a positive solution $x$ such that
\[
   0<x(t)\le\frac{1}{c_m}\left(\frac{e^*}{c}\right)^\frac{1}{\mu}
   \quad\mbox{for all \ } t\in [0,T].
\]
\end{theorem}
\begin{proof} We will verify that the assumptions of Theorem \ref{main} are satisfied.

\noindent {\it Claim 1. \thetag{H0} and \thetag{H1} are true.}

Clearly, \thetag{H0} is a consequence of \thetag{C0}. Furthermore, let $m^2$ be as in
\thetag{C6} and let
\[
   R_2=\displaystyle\frac{1}{c_m}\left(\frac{e^*}{c}\right)^\frac{1}{\mu}.
\]
We have
\[
   f_m(t,x)=\frac{e(t)}{\mu}\,x^{1-2\mu}-\frac{c}{\mu}\,x^{1-\mu}+m^2 x
   \quad\mbox{for \ } t\in[0,T] \mbox{ \ and \ } x\ge 0.
\]
Notice that
\begin{equation}\label{r1}
\left\{
  \begin{array}{l}
   f_m(t,x)\ge f_*(x)\quad\mbox{for all \ } t\in[0,T] \mbox{\ and \ } x\ge 0,\mbox{ \ where}
  \\[2mm]\displaystyle
   f_*(x)=\frac{e_*}{\mu}\,x^{1-2\mu}-\frac{c}{\mu}\,x^{1-\mu}+m^2 x
         =\frac{x^{1-2\mu}}{\mu}\,\left(e_*-c\,x^{\,\mu}+m^2 \mu\,x^{2\mu}\right)\,.
  \end{array}
\right.
\end{equation}
Obviously, $f_*(x)\ge 0$ whenever $x\in[0,\left(\frac{e_*}{c}\right)^\frac{1}{\mu}].$
Let $x\in[\left(\frac{e_*}{c}\right)^\frac{1}{\mu},R_2].$ Taking into account \thetag{C6}, we get
\begin{align*}
  \mu\,x^{2\mu-1}f_*(x)&=e_*-c\,x^{\,\mu}+m^2\mu\,x^{2\mu}
  \ge e_*-c\,R_2^{\,\mu}+m^2\mu\left(\left(\frac{e_*}{c}\right)^\frac{1}{\mu}\right)^{2\mu}
 \\
  &=e_*-\frac{e^*}{c_m^{\,\mu}}+m^2 \mu\left(\frac{e_*}{c}\right)^{2}
   =\mu\left(\frac{e_*}{c}\right)^{2}\left(m^2-\frac{c^2}{\mu\,e_*^2}
    \left(\frac{e^*}{c^\mu_m}-e_*\right)\right)
   \ge 0.
\end{align*}
Thus $f_*(x)\ge 0$ for all $x\in[0,R_2]$ and hence
\[
  f_m(t,x)\ge 0 \quad\mbox{for \ } t\in[0,T] \mbox{ \ and \ } x\in[0,R_2].
\]
In particular, we conclude that \thetag{H1} is true (with an arbitrary $R_1\in(0,R_2)$).

\vskip2.5mm

\noindent {\it Claim 2. \thetag{H2} and \thetag{H4} are true.}

First, we will prove that $f_m(t,x)\le m^2 R_2$ for $t\in[0,T]$ and
$x\in[c_m\,R_2,R_2].$ Indeed, for $t\in[0,T]$ and $x\ge 0$ we have
\[
   f_m(t,x)\le\frac{e^*}{\mu}\,x^{1-2\mu}-\frac{c}{\mu}\,x^{1-\mu}+m^2x\,.
\]
Furthermore, if $x\ge c_m\,R_2=\left(\dfrac{e^*}{c}\right)^\frac{1}{\mu},$ then
\[
   \frac{e^*}{\mu}\,x^{1-2\mu}-\frac{c}{\mu}\,x^{1-\mu}
   =\frac{x^{1-2\mu}}{\mu}\,(e^*-c\,x^{\,\mu})
   \le\frac{x^{1-2\mu}}{\mu}\left(e^*-c\,\frac{e^*}{c}\right)=0\,.
\]
Therefore, for $t\in[0,T]$ and  $x\in[c_m\,R_2,R_2]$ we have
\[
   f_m(t,x)\le\frac{e^*}{\mu}\,x^{1-2\mu}-\frac{c}{\mu}\,x^{1-\mu}+m^2x\le m^2x\le m^2R_2.
\]
Now, it is easy to show (see also Remark \ref{rem-main}) that \thetag{H2} and \thetag{H4}
are true with $g_1(t)=m^2 R_2.$

\vspace{0.25cm}
\noindent {\it Claim 3. \thetag{H5} and \thetag{H6} are true.}

First, we will show that there is $R_1\in (0,R_2)$ such that $f_m(t,x)\ge m^2 R_1$ for
$t\in[0,T]$ and $x\in [c_m R_1,R_1].$ Indeed, due to \thetag{C5} we can choose $R_1\in (0,R_2)$
in such a way that
\begin{equation}\label{eqR1}
   (1-c_m)\,m^2 \mu R_1^{2\mu}+c  R_1^{\,\mu}\le e_*\,c_m^{1-2\mu}.
\end{equation}
On the other hand, for $t\in[0,T]$ and $x\in[c_m\,R_1,R_1]$ we deduce
\begin{align*}
  f_*(x)&=\frac{e_*}{\mu}x^{1-2\mu}-\frac{c}{\mu}x^{1-\mu}+m^2x
  \ge\frac{e_*}{\mu}(c_m\,R_1)^{1-2\mu}-\frac{c}{\mu}R_1^{1-\mu}+m^2c_m\,R_1
 \\
  &=\frac{R_1^{1-2\mu}}{\mu}(e_*\,c_m^{1-2\mu}-c\,R_1^{\,\mu}+c_m m^2 \mu R_1^{2\mu})
 \\\noalign{\noindent\mbox{and hence, by \eqref{eqR1},}}
  f_*(x)&\ge m^2 R_1,
\end{align*}
wherefrom, according to \eqref{r1}, the desired inequality immediately follows. Now, it is easy
to show (see also Remark \ref{rem-main}) that \thetag{H5} and \thetag{H6} are true with
$g_0(t)=m^2R_1.$

\vspace{2mm}

By Claims 1--3, all the assumptions of Theorem \ref{main} are satisfied and the proof
can be completed applying it.
\end{proof}

\begin{example}\label{ex-mod} Let us modify \cite[Example 3.7]{citz} in order to admit
a non-constant function $e.$ Indeed,  consider problem \eqref{eqprobreg} with $a=1.6,$
$\mu=0.01,$ $c=1.49,$ $T=1$ and $e$ a continuous periodic function such that $e_*=1.54.$
Notice that in this case we cannot apply \cite[Theorem 1.8]{cpt} since
\[
 \frac{\,c^2}{4\,\mu\,e_*}\approx 36.0406>10.5096\approx\left(\frac{\pi}{T}\right)^2{+}\,\frac{a^2}4\,.
\]
On the other hand, for $m=0.7$ we have $c_m\approx 0.9414$ and then Theorem~\ref{th}
implies that this problem has a positive solution $x$ provided that $e^*<1.5443.$ Moreover,
\[
   0<x(t)\le\frac{1}{c_m}\left(\frac{e^*}{c}\right)^\frac{1}{\mu}<38.0844
   \quad\mbox{for all \ } t\in [0,T].
\]

Finally we show some numerical examples for particular choices of $e$:

\medbreak
\vspace*{0.25cm}
\begin{minipage}[b]{0.475\linewidth}
\includegraphics[scale=0.77]{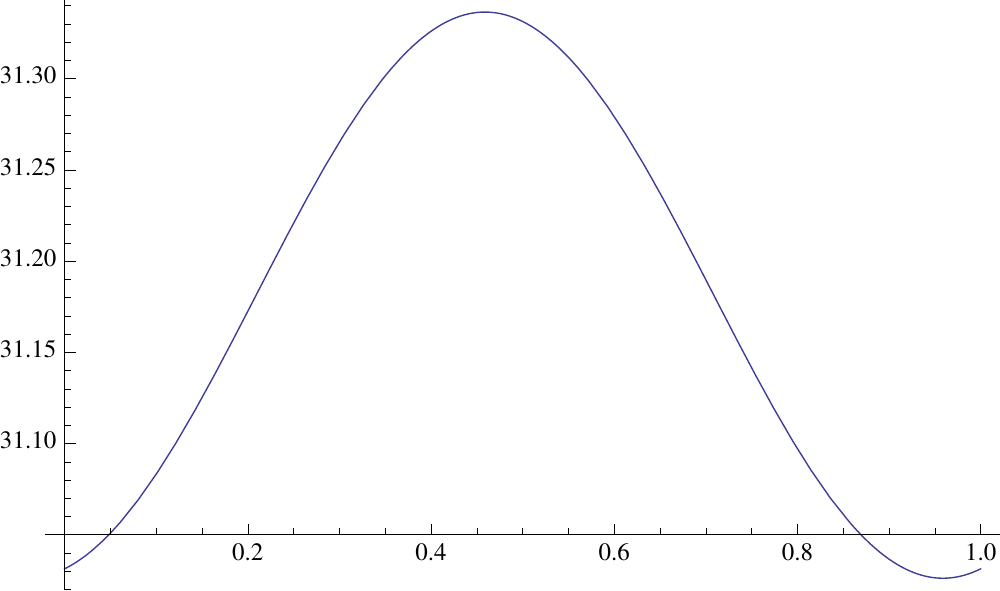}
\begin{center}
Solution of problem \eqref{eqprobreg} with $a=1.6,$ $\mu=0.01,$ $c=1.49,$ $T=1$ and
$e(t)=1.54215+0.002097\cos(2\pi\,t)$
\end{center}
\end{minipage} \hfill
\begin{minipage}[b]{0.475\linewidth}
\includegraphics[scale=0.77]{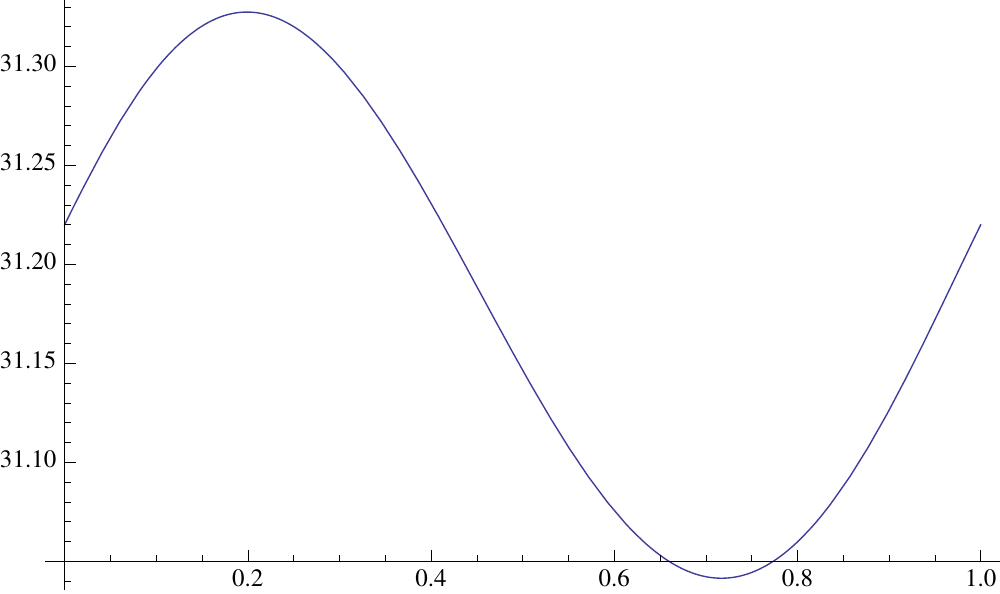}
\begin{center}
Solution of problem \eqref{eqprobreg} with $a=1.6,$ $\mu=0.01,$ $c=1.49,$ $T=1$ and
$e(t)=1.54215-0.02 (t-3 t^2+2 t^3)$
\end{center}
\end{minipage}
\end{example}
Some computations in Examples \ref{ex-neg} and \ref{ex-mod} were made with the help of the software
system \textsl{Mathematica}.

\section*{Acknowledgements}
We would like to thank an anonymous referee for his/her careful reading of the manuscript and the many remarks that improved the quality of the paper. In particular his/her insights lead us to example \eqref{example}.

J. A. Cid was partially supported by Ministerio de Educaci\'on y Ciencia, Spain, and FEDER,
Project MTM2013-43404-P,
G. Infante was partially supported by G.N.A.M.P.A. - INdAM (Italy),
M. Tvrd\'y was supported by GA \v{C}R Grant P201/14-06958S and  RVO: 67985840 and
M. Zima was partially supported by the Centre for Innovation and Transfer of Natural Science
and Engineering Knowledge of University of Rzesz\'ow.
This research was completed when M. Zima was visiting Departamento de Matem\'aticas,
Universidade de Vigo, Campus de Ourense, and Mathematical Institute,  Czech Academy
of Sciences in April, May and June 2016; the kind hospitality is gratefully acknowledged.


\begin{thebibliography}{99}

\bibitem{liebau} G.~Liebau, \"Uber ein ventilloses Pumpprinzip,
Naturwissenschaften 41 (1954) 327.

\bibitem{and} H.~Andersson, W.~van~der~Wijngaart, P.~Nilsson, P.~Enoksson and G.~Stemme,
A valve-less diffuser micropump for microfluidic analytical systems,
Sensor. Actuat. B-Chem. 72 (2001) 259--265.

\bibitem{borzipropst} A.~Borz\`i and G.~Propst,
Numerical investigation of the Liebau phenomenon,
Z. Angew. Math. Phys. 54 (2003) 1050--1072.

\bibitem{bringley} T.T.~Bringley, S.~Childress, N.~Vandenberghe and J.~Zhang,
An experimental investigation and a simple model of a valveless pump,
Phys. Fluids 20 (2008), 033602, 15 pages.

\bibitem{mwy} J.~M\"anner, A.~Wessel and T.M.~Yelbuz,
How does the tubular embryonic heart work? Looking for the physical mechanism generating unidirectional
blood flow in the valveless embryonic heart tube,
Developmental Dynamics 239 (4) (2010) 1035--1046;
\url{http://onlinelibrary.wiley.com/doi/10.1002/dvdy.22265/full}.

\bibitem{propst} G.~Propst, Pumping effects in models of periodically forced flow configurations,
Phys. D 217 (2006) 193--201.

\bibitem{torres} P.J.~Torres,
Mathematical models with singularities - A Zoo of Singular Creatures,
Atlantis Briefs in Differential Equations, vol.~1, Atlantis Press, Paris, 2015.

\bibitem{cpt} J.A.~Cid, G.~Propst and M.~Tvrd\'y, On the pumping effect in a pipe/tank flow
configuration with friction, Phys. D 273-274 (2014) 28--33.

\bibitem{citz} J.A.~Cid, G.~Infante, M.~Tvrd\'y and M.~Zima, A topological approach to periodic
oscillations related to the Liebau phenomenon,
J. Math. Anal. Appl. 423 (2015) 1546--1556.

\bibitem{gippmt} G.~Infante, P.~Pietramala and M.~Tenuta, Existence and localization of positive
solutions for a nonlocal BVP arising in chemical reactor theory,
Commun. Nonlinear Sci. Numer. Simulat. 19 (2014) 2245--2251.

\bibitem{lan01} K.Q.~Lan, Multiple positive solutions of semilinear differential
equations with singularities, J. London. Math. Soc 63 (2001) 690--704.

\bibitem{klamc04} K.Q.~Lan, Multiple positive solutions of Hammerstein integral equations
and applications to periodic boundary value problems,
Appl. Math. Comput. 154 (2004) 531--542.

\bibitem{li} Y.X.~Li, Existence of positive solutions for the cantilever beam equations,
Nonlinear Anal. Real World Appl. 27 (2016) 221--237.

\bibitem{jeff-2} J.R.L.~Webb, Existence of positive solutions for a thermostat model,
Nonlinear Anal. Real World Appl. 13 (2012) 923--938.

\bibitem{jeff} J.R.L.~Webb, Positive solutions of nonlinear equations via comparison with
linear operators, Discrete Contin. Dyn. Syst.33 (2013) 5507--5519.

\bibitem{omaritromb} P.~Omari and M.~Trombetta,
Remarks on the lower and upper solutions method for second and third-order periodic boundary
value problems, Appl. Math. Comput. 50 (1992) 1--21.

\bibitem{krzab} M.A.~Krasnosel'ski\u\i{} and P.P.~Zabre\u\i{}ko,
Geometrical methods of nonlinear analysis, Springer-Verlag, Berlin, 1984.

\bibitem{guolak} D.~Guo and V.~Lakshmikantham, Nonlinear Problems in Abstract Cones,
Academic Press, Boston, 1988.

\bibitem{amann} H.~Amann, Fixed point equations and nonlinear eigenvalue problems in ordered
Banach spaces, SIAM Rev. 18 (1976) 620--709.

\bibitem{deimling} K.~Deimling, Nonlinear Functional Analysis, Springer, New York, 1985.

\bibitem{zeidler} E.~Zeidler, Nonlinear Functional Analysis and Its Applications,
Springer, New York, 1993.
\end{thebibliography}
\end{document}